\documentclass[12pt]{amsproc}
\newtheorem{theorem}{\sc Theorem}[section]
\newtheorem{lemma}[theorem]{\sc Lemma}

\begin{document}
\author[P. Shumyatsky]{Pavel Shumyatsky}
\address{Department of Mathematics, University of Brasilia, 70910-900 Bras\'ilia DF, Brazil}
\email{pavel@unb.br}

\title[Groups with many elements of prime power order]{Profinite groups in which many elements have prime power order}

\keywords{Profinite groups, centralizers}
\subjclass[2010]{20E18}
\thanks{This research was supported by CNPq and FAPDF}

\maketitle

\begin{abstract}
The structure of finite and locally finite groups in which every element has prime power order (CP-groups) is well known. In this paper we note that the combination of our earlier results with the available information on the structure of finite CP-groups yields a detailed description of profinite groups with that property (Theorem \ref{main0}). Then we deal with two generalizations of profinite CP-groups.

\noindent Theorem \ref{main1}. A profinite group $G$ is virtually pro-$p$ for some prime $p$ if and only if for each nontrivial $x\in G$ there is a prime $p$ (depending on $x$) such that $C_G(x)$ is virtually pro-$p$. 

\noindent Theorem \ref{main2}. Let $G$ be a profinite group in which each element has either finite or prime power (possibly infinite) order. Then $G$ is either torsion or virtually pro-$p$ for some prime $p$.
\end{abstract}

\section{Introduction}

Groups all of whose elements have prime power order for brevity are called CP-groups. Of course these are precisely the groups in which the centralizer of every nontrivial element consists of elements of prime power order. G. Higman considered finite soluble CP-groups in 1957. He showed that any such group is a $p$-group, or Frobenius, or 2-Frobenius, and its order has at most two distinct prime divisors \cite{higman}. Later Suzuki determined all nonabelian finite simple CP-groups \cite[Theorem 16]{suzuki}. A detailed description of finite CP-groups was given by Brandl in \cite{brandl} while that of locally finite CP-groups by A. L. Delgado and Yu Fen Wu in \cite{delgado} (see also Heineken \cite{heineken}).

Regarding profinite groups, it is customary to say that an element $x$ of a profinite group $G$ has prime power order if the image of $x$ has prime power order in every finite continuous homomorphic image of $G$ (see \cite[p. 32]{rz} for a formal definition of the order of a profinite group). Throughout this paper we say that a profinite group has a property virtually if it has an open subgroup with that property. In \cite{acn} we studied profinite groups in which all centralizers of nontrivial elements are pronilpotent. It was shown that such groups are virtually pronilpotent. Now, combining this with the available information on finite CP-groups it is easy to deduce the following theorem.
\bigskip

\begin{theorem}\label{main0} Let $G$ be a profinite CP-group. Then $G$ is virtually pro-$p$ for some prime $p$. If $G$ is infinite and $P$ is the maximal open normal pro-$p$ subgroup in $G$, then either $G=P$ or   one of the following occurs.
\begin{enumerate}
\item $G$ is a profinite Frobenius group whose kernel is $P$ and complement is either cyclic of prime power order or a (generalized) quaternion group.
\item There is an odd prime $q\neq p$ such that $G/P$ is a Frobenius group with cyclic kernel of $q$-power order and cyclic complement of $p$-power order. In this case $G$ is a so-called $3$-step group (see \cite[p. 401]{go}).
\item $G/P$ is isomorphic to one of the four simple groups $L_2(4)$, $L_2(8)$, $Sz(8)$, $Sz(32)$ and $P$ is (topologically) isomorphic to a Cartesian sum of natural modules for $G/P$.
\end{enumerate}
\end{theorem}

Our main objective in this paper is to study profinite groups satisfying certain CP-conditions. The paper contributes to the line of research dealing with profinite groups in which the centralizers of nontrivial elements have prescribed structure (cf. \cite{ppz,acn,abac}). In particular, we consider groups $G$ in which for every nontrivial element $x\in G$ there is a prime $p$, possibly depending on $x$, such that the centralizer $C_G(x)$ is virtually pro-$p$. It turns out that if $G$ is infinite then there is a unique prime $p$ such that the centralizer $C_G(x)$ is virtually pro-$p$ for all $x\in G$.

\begin{theorem}\label{main1} A profinite group $G$ is virtually pro-$p$ for some prime $p$ if and only if for each nontrivial $x\in G$ there is a prime $p$ such that $C_G(x)$ is virtually pro-$p$. 
\end{theorem}

Next, we consider profinite groups in which all non-torsion elements have prime power order. This is equivalent to saying that the centralizers of elements of infinite order in $G$ are pro-$p$ groups. Obvious examples of such groups are profinite torsion groups and pro-$p$ groups. Further examples are provided by profinite Frobenius groups whose kernel is a pro-$p$ group. We do not know whether there are other types of examples of groups in question.

\begin{theorem}\label{main2} Let $G$ be a profinite group in which each element has either finite or prime power (possibly infinite) order. Then $G$ is either torsion or virtually pro-$p$ for some prime $p$.
\end{theorem}

We remark that the combination of Theorem \ref{main2} and Lemma \ref{qqq} (see the next section) implies that if $G$ is not torsion, then there is a prime $p$ such that $G/O_p(G)$ is finite with Sylow subgroups either cyclic or (generalized) quaternion. Here $O_p(G)$ stands for the maximal normal pro-$p$ subgroup of $G$. Hence, the structure of $G/O_p(G)$ is similar to that of a complement in a finite Frobenius group. A detailed description of that structure can be found in Passman \cite[Ch. 3]{passman}.

In the next section we collect auxiliary results required in the proofs of the above theorems. Then the proofs are given in the subsequent sections.

\section{Preliminaries}

Recall that a group is said to locally have a property if every finitely generated subgroup has that property. The next lemma is almost obvious. We include the proof for the reader's convenience. Note that the lemma is no longer true if we drop the assumption that $G$ is residually finite (cf.\ the non-abelian semidirect product of the Pr\"ufer group $C_{2^\infty}$ by the group of order 2). Throughout the paper we denote by $\langle X\rangle$ the (sub)group generated by a set $X$.

\begin{lemma}\label{114} 
Let $G$ be a locally nilpotent group containing an element with finite centralizer. Suppose that $G$ is residually finite. Then $G$ is finite.
\end{lemma}
\begin{proof} Choose $x\in G$ such that $C_G(x)$ is finite. Let $N$ be a normal subgroup of finite index such that $N\cap C_G(x)=1$. Assume that $N\neq1$ and let $1\neq y\in N$. The subgroup $\langle x,y\rangle$ is nilpotent and so the center of $\langle x,y\rangle$ has nontrivial intersection with $N$. This is a contradiction since $N\cap C_G(x)=1$. The result follows.
\end{proof}

We will require the theorems of Wilson \cite{wilson} and Zelmanov \cite{ze} on the structure of compact torsion groups.

\begin{theorem} \label{wilson} (Wilson) Let $G$ be a compact torsion group. Then $G$ has a finite series of characteristic subgroups, in which each factor either is a pro-$p$ group for some prime $p$ or is isomorphic (as a topological group) to a Cartesian product of isomorphic finite simple groups. 
\end{theorem}
Remark that the above theorem generalizes an earlier result of Herfort saying that the order of a compact torsion group is divisible by only finitely many primes \cite{herfort}. It is noteworthy that later Herfort discovered that this holds true for any profinite group $G$ in which all elements have order divisible by only finitely many primes \cite{herf2}. In particular the latter result of Herfort implies that the order of any (generalized) profinite CP-group is divisible by only finitely many primes. In the present paper however we will not use this fact.

\begin{theorem}\label{torsion} (Zelmanov) Each compact torsion group is locally finite.
\end{theorem}

Recall that according to the Hall-Kulatilaka theorem \cite{haku} each infinite locally finite group has an infinite abelian subgroup. Combining this with Theorem \ref{torsion} we deduce

\begin{theorem}\label{infab} Each infinite profinite or locally finite group has an infinite abelian subgroup.
\end{theorem}

If $A$ is a group of automorphisms of a group $G$, the subgroup generated by all elements of the form $g^{-1}g^\alpha$ with $g\in G$ and $\alpha\in A$ is denoted by $[G,A]$. It is well known that the subgroup $[G,A]$ is an $A$-invariant normal subgroup in $G$. We write $C_G(A)$ for the centralizer  of $A$ in $G$. The symbol $A^{\#}$ stands for the set of nontrivial elements of the group $A$. By an automorphism of a profinite group we always mean a continuous automorphism. As usual, $\pi(G)$ denotes the set of prime divisors of the order of $G$.

The next lemma is a list of useful facts on coprime actions. Here $|K|$ means the order of a profinite group $K$. For finite groups the lemma is well known (see for example \cite[Ch.~5 and 6]{go}). For infinite profinite groups the lemma follows from the case of finite groups and the inverse limit argument (see \cite[Proposition 2.3.16]{rz} for a detailed proof of item (iii)). Part (vi) is straightforward from claims (iv) and (v).

\begin{lemma}\label{cc}
Let  $A$ be a profinite group of automorphisms of a profinite group $G$ such that $(|G|,|A|)=1$. Then
\begin{enumerate}
\item[(i)] $G=[G,A]C_{G}(A)$.
\item[(ii)] $[G,A,A]=[G,A]$. 
\item[(iii)] $C_{G/N}(A)=NC_G(A)/N$ for any $A$-invariant normal subgroup $N$ of $G$.
\item[(iv)] If $G$ is pronilpotent and $A$ is a noncyclic abelian group, then $G=\prod_{a\in A^{\#}}C_{G}(a)$.
\item[(v)] $G$ contains an $A$-invariant Sylow $q$-subgroup for each prime $q\in\pi(G)$.
\item[(vi)] If $A$ is a noncyclic abelian group and $C_G(a)$ is finite for each $a\in A^{\#}$, then $G$ is finite.
\end{enumerate}
\end{lemma}

A proof of the next lemma can be found in \cite{abac}.
\begin{lemma}\label{resh} Let $A$ be an infinite procyclic group acting coprimely on a profinite group $G$. Let $A_i=A^{i!}$ denote the subgroup generated by the $i!$th powers of elements of $A$. Set $G_i=C_G(A_i)$. For each $A$-invariant open normal subgroup $N$ of $G$ there is $i$ such that $G=NG_i$. If there is an index $j$ such that $G_j=G_{j+k}$ for each $k=1,2\dots$, then $G_j=G$.
\end{lemma}

We will also require the following result from \cite{qj} whose proof uses a deep Lie-theoretic result of Zelmanov \cite{ze17}.

\begin{lemma}\label{qqq} Let $p$ be a prime and $A$ the noncyclic group of order $p^2$. Suppose that $A$ acts by automorphisms on a pro-$p'$ group $G$ in such a manner that $C_G(a)$ is torsion for each $a\in A^\#$. Then $G$ is locally finite. In particular, $G$ is torsion.
\end{lemma}

The next lemma uses the well-known corollary of the classification of finite simple groups that the order of any nonabelian finite simple group is divisible by 3 or 5.

\begin{lemma}\label{prost} Let $H$ be a profinite group in which for all odd primes $p\in\pi(G)$ the Sylow $p$-subgroups are finite. Then $H$ is virtually prosoluble.
\end{lemma}
\begin{proof} Choose an open $\{3,5\}'$-subgroup $K\leq G$. We see that $K$ is prosoluble, as required.
\end{proof}

\section{Proof of Theorem \ref{main0}}

This section is quite short. We just briefly explain why Theorem \ref{main0} is fairly straightforward from \cite{acn} and earlier results on finite CP-groups.

Recall that Theorem \ref{main0} deals with a profinite CP-group $G$. Since every element of $G$ has prime power order, it follows that the centralizers of nontrivial elements in $G$ are pronilpotent and so by \cite{acn} the group $G$ is virtually pronilpotent. Assume that $G$ is infinite and observe that the order of the maximal normal pronilpotent subgroup of $G$ is divisible by exactly one prime, say $p$. Therefore $G$ is virtually pro-$p$. Note that the continuous finite images of $G$ are CP-groups.

Let $P$ be the maximal normal pro-$p$ subgroup of $G$. Suppose that $G$ is not a pro-$p$  group (that is, $G\neq P$) and assume first that $G$ is prosoluble. Higman's work \cite[Theorem 1]{higman} immediately implies that $G$ is either a profinite Frobenius group whose kernel is $P$ and complement has prime power order (in which case the complement is either cyclic or quaternion) or there is an odd prime $q\neq p$ such that $G/P$ is a Frobenius group with cyclic kernel of $q$-power order and cyclic complement of $p$-power order. Moreover, if in the latter case $Q$ denotes a Sylow $q$-subgroup of $G$, note that $PQ$ is also a Frobenius group.

We will therefore assume that $G$ is not prosoluble. It was observed that in this case $p=2$ (see for example \cite{acn}; this can also be deduced from the classical results of Suzuki, cf. \cite[Part II, Theorem 1]{suzuki0}). Let $N$ be an open normal subgroup of $G$ contained in $P$. The finite quotient $G/N$ is a nonsoluble finite CP-group whose Fitting subgroup is nontrivial. By \cite{delgado} $G/P$ is isomorphic to one of the four simple groups $L_2(4)$, $L_2(8)$, $Sz(8)$, $Sz(32)$ and $P/N$ is isomorphic to a direct sum of natural modules for $G/P$. Since the group $G$ is an inverse limit of finite groups $G/N$ where $N$ ranges over open normal subgroups contained in $P$, it follows that $G/P$ is isomorphic to one of the four simple groups and $P$ is (topologically) isomorphic to a Cartesian sum of natural modules for $G/P$. The proof is now complete. \hfill $\qed$
\bigskip

An immediate consequence of Theorem \ref{main0} is that if $G$ is a profinite CP-group, then $G$ is either pro-$p$ or virtually nilpotent. This is because a kernel in a profinite Frobenius group is always nilpotent \cite[Corollary 4.6.10]{rz}.

\section{Theorem \ref{main1}}

Obviously, if $G$ is a virtually pro-$p$ group, then every centralizer in $G$ is virtually pro-$p$. We need to prove the converse of this statement.

Thus, in this section $G$ is a profinite group in which for each nontrivial element $x\in G$ there is a prime $p$ (depending on $x$) such that $C_G(x)$ is virtually pro-$p$. Our objective is to show that $G$ is virtually pro-$p$. We write $O_\pi(G)$ for the maximal normal pro-$\pi$ subgroup of $G$. 

\begin{lemma}\label{nilp2} Assume that $G$ is pronilpotent. Then $G$ is virtually pro-$p$ for some prime $p$.
\end{lemma}
\begin{proof} We can assume that $G$ is infinite. Suppose first that each Sylow subgroup of $G$ is finite. Then $C_G(x)$ is finite for each $x\in G$. This leads to a contradiction because of Theorem \ref{infab}. 

Hence, $G$ has an infinite Sylow $r$-subgroup $R$ for some $r\in\pi(G)$. Write $G=R\times Q$, where $Q=O_{r'}(G)$. Suppose that $Q$ is infinite. By Theorem \ref{infab} $Q$ contains an infinite abelian subgroup $A$. Observe that $R\times A$ is contained in $C_G(x)$ for each $x\in A$. Thus, $C_G(x)$ is not virtually pro-$p$ for any $p\in\pi(G)$, a contradiction. Therefore $Q$ is finite and the lemma follows.
\end{proof}
Throughout the paper we write $\Bbb Z_p$ to denote the infinite procyclic pro-$p$ group. This is of course isomorphic to the additive group of $p$-adic integers.
\begin{lemma}\label{infpq2} Let $p$ be a prime and $K$ an infinite pro-$p'$ subgroup of $G$. The pro-$p$ subgroups of $N_G(K)$ are finite.
\end{lemma}
\begin{proof} Suppose that the lemma is false and there is an infinite pro-$p$ subgroup $A$ in $N_G(K)$. In view of Theorem \ref{infab}, without loss of generality we can assume that $A$ is abelian. Since $A$ is infinite, $C_G(a)$ must be virtually pro-$p$ for each nontrivial $a\in A$. It follows that $C_K(a)$ is finite for each nontrivial $a\in A$.

Let us show that $A$ is torsion. Assume that this is false and $A$ contains a subgroup isomorphic to $\Bbb Z_p$. We can simply assume that $A\cong\Bbb Z_p$. Choose a generator $a$ of $A$. Set $a_i=a^{p^i}$ and $K_i=C_K(a_i)$ for $i=1,2,\dots$. We know that all subgroups $K_i$ are finite. By Lemma \ref{resh} the subgroup $H=\bigcup_iK_i$ is an infinite locally finite group. Theorem \ref{infab} tells us that $H$ contains an infinite abelian subgroup $L$. Choose a nontrivial element $l\in L$. There exists an index $i$ such that $l\in K_i$. We see that both subgroups $L$ and $\langle a_i\rangle$ are contained in $C_G(l)$ and so $C_G(l)$ is not virtually pro-$q$ for any prime $q$.

Hence, $A$ is torsion. Since $A$ is infinite, there is an elementary abelian subgroup $B<A$ of order $p^2$. Taking into account that $C_K(b)$ is finite for each nontrivial $b\in B$ we deduce from Lemma \ref{cc} (vi) that $K$ is finite, a contradiction.
\end{proof} 

\begin{lemma}\label{finsyl} Suppose that the Sylow $p$-subgroups of $G$ are finite for each prime $p\in\pi(G)$. Then $G$ is finite.
\end{lemma}
\begin{proof} Since the centralizers in $G$ are virtually pro-$p$ and since all Sylow $p$-subgroups of $G$ are finite, we conclude that $C_G(x)$ is finite for each nontrivial $x\in G$. In view of Theorem \ref{infab}, $G$ is finite.
\end{proof}
\begin{lemma}\label{infsy} Suppose that $G$ is infinite. There is a unique prime $p\in\pi(G)$ such that the Sylow $p$-subgroups of $G$ are infinite.
\end{lemma}
\begin{proof} Assume that the lemma is false. In view of Lemma \ref{finsyl}, there are two different primes $p$ and $q$ for which the Sylow subgroups of $G$ are infinite. Let $P$ be a Sylow $p$-subgroup and $Q$ a Sylow $q$-subgroup of $G$. Since the subgroups $P$ and $Q$ are infinite, we can choose an infinite chain of open normal subgroups $G=N_1>N_2>\dots$ such that $P\cap N_i>P\cap N_{i+1}$ and $Q\cap N_i>Q\cap N_{i+1}$ for each $i=1,2,\dots$. 

Set $P_i=P\cap N_i$ and $H_i=N_G(P_i)$. The Frattini argument shows that $G=N_iH_i$ for each $i=1,2,\dots$. Lemma \ref{infpq2} implies that the Sylow $q$-subgroups $Q_i$ of $H_i$ are finite. Obviously, the subgroups $Q_i$ can be chosen in such a way that $Q_1\leq Q_2\leq\dots$. Let $Q_0=\cup_iQ_i$ and $N=\cap_iN_i$. The choice of the chain $G=N_1>N_2>\dots$ guarantees that $Q\cap N$ has infinite index in $Q$. Therefore $Q_0$ is an infinite (abstract) subgroup of $G$. Further, we note that $Q_0$ is locally nilpotent and so, by Lemma \ref{114}, $C_{Q_0}(a)$ is infinite for each $a\in Q_0$. It follows that $C_G(a)$ is virtually pro-$q$ for each $a\in Q_0$ and therefore $C_P(a)$ is finite. Since $Q_0$ is infinite, there is an elementary abelian subgroup $B<Q_0$ of order $q^2$. The construction of $Q_0$ guarantees that $B$ normalizes an open subgroup $P^*$ of $P$. Taking into account that $C_P(b)$ is finite for each nontrivial $b\in B$ we deduce from Lemma \ref{cc} (vi) that $P$ is finite, a contradiction.
\end{proof}

The remaining part of the proof of Theorem \ref{main1} will be easy.
\begin{proof}[Proof of Theorem \ref{main1}] Recall that $G$ is a profinite group in which for each $x\in G$ there is a prime $p$ (depending on $x$) such that $C_G(x)$ is virtually pro-$p$. We need to prove that $G$ is virtually pro-$p$. Note first that $G$ is virtually prosoluble. Indeed, by Lemma \ref{infsy}, there is at most one prime $p\in\pi(G)$ such that the Sylow $p$-subgroups of $G$ are infinite. If the Sylow 2-subgroups are finite, then $G$ has an open normal 2$'$-subgroup which is prosoluble by the Feit-Thompson theorem \cite{fetho}. If $G$ has infinite Sylow 2-subgroups, then it is  virtually prosoluble by Lemma \ref{prost}. 

Thus, we know that $G$ is virtually prosoluble. It is sufficient to prove the theorem in the case where $G$ is prosoluble. Recall that in view of Lemma \ref{finsyl} there is exactly one prime  $p\in\pi(G)$ for which the Sylow $p$-subgroups of $G$ are infinite. 

Let $H$ be a Hall $p'$-subgroup of $G$. Lemma \ref{finsyl} shows that $H$ is finite. The group $G$ contains an open normal subgroup $N$ such that $N\cap H=1$. It follows that $N$ is a pro-$p$ group. Therefore $G$ is virtually pro-$p$, as required.
\end{proof}

\section{Theorem \ref{main2}}

In this section we prove Theorem \ref{main2}. Thus, our goal is to show that if $G$ a profinite group in which each element has either finite or prime power (possibly infinite) order, then  $G$ either is torsion or virtually pro-$p$ for some prime $p$. In particular, if $G$ is finitely generated, then $G$ is virtually pro-$p$.

\begin{lemma}\label{nilp2} Assume that the group $G$ in Theorem \ref{main2} is pronilpotent. Then $G$ either is torsion or pro-$p$ for some prime $p$.
\end{lemma}
\begin{proof} Assume that $G$ is not torsion. It follows that $G$ contains a subgroup $A$ isomorphic to $\Bbb Z_p$ for some prime $p$. By the hypothesis, the centralizer $C_G(A)$ must be a pro-$p$ group. Since $G$ is pronilpotent, we deduce that $G$ is a pro-$p$ group. Hence the result.
\end{proof}

\begin{lemma}\label{cart2} Assume that the group $G$ in Theorem \ref{main2} is topologically isomorphic to a Cartesian product of finite simple groups. Then $G$ has finite exponent.
\end{lemma}
\begin{proof} We can assume that $G$ is infinite. Write $G=\prod_{i\in\Bbb N} S_i$ where $S_i$ are nonabelian finite simple groups. Suppose that the lemma is false. We can choose an element $x=(a_1,a_2,\dots,a_i,\dots)$ with $a_i\in S_i$ such that the orders of the elements $a_i$ are unbounded. Thus, by hypotheses, $x$ must be a $p$-element. Now choose a $p'$-element $b\in S_1$ and set $y=(b,a_2,\dots,a_i,\dots)$. We see that $y$ has infinite order and is not a $p$-element, contrary to the hypotheses.
\end{proof}

We will require the following analogue of Lemma \ref{infpq2}.

\begin{lemma}\label{infpq3} Assume the hypotheses of Theorem \ref{main2}. Let $p$ and $q$ be distinct primes and $K$ a non-torsion pro-$p$ subgroup of $G$. The pro-$q$ subgroups of $N_G(K)$ are finite, either cyclic or (generalized) quaternion.
\end{lemma}
\begin{proof} Let $Q$ be a Sylow $q$-subgroup of $N_G(K)$. Suppose first that $Q$ is a non-torsion and so $Q$ contains a subgroup $A\cong\Bbb Z_p$. By the hypotheses, $C_G(a)$ must be a pro-$q$ group for each $1\neq a\in A$. In particular $C_K(a)=1$ for each $1\neq a\in A$. Because of Lemma \ref{resh} this is a contradiction.

Thus, $Q$ is torsion and so, by Theorem \ref{torsion}, $Q$ is locally finite. Suppose that $Q$ contains an elementary abelian subgroup $B$ of order $q^2$. The hypotheses imply that $C_K(b)$ is torsion for each nontrivial $b\in B$. Now Lemma \ref{qqq} tells us that $K$ is torsion, a contradiction.

Hence, $Q$ is a locally finite $q$-group which does not have noncyclic subgroups of order $q^2$. It follows that $Q$ is finite, either cyclic or (generalized) quaternion.
\end{proof} 

Recall that the Fitting height of a finite soluble group $G$ is the length $h(G)$ of a shortest series of normal subgroups all of whose quotients are nilpotent. By the Fitting height of a prosoluble group $G$ we mean the length $h(G)$ of a shortest series of normal subgroups all of whose quotients are pronilpotent. Note that in general a prosoluble group does not necessarily have such a series. The parameter $h(G)$ is finite if, and only if, $G$ is an inverse limit of finite soluble groups of bounded Fitting height. The following lemma is taken from \cite{acn}.

\begin{lemma}\label{44} Let $p$ be a prime and $G$ a finite soluble group in which for every prime $q\neq p$ the Sylow $q$-subgroups are cyclic or generalized quaternion. The Fitting height of $G$ is at most $4$.
\end{lemma}

\begin{lemma}\label{2244} Assume the hypotheses of Theorem \ref{main2} and suppose that $G$ is not torsion. If $H$ is a prosoluble subgroup of $G$, then $h(H)\leq4$.
\end{lemma}
\begin{proof} Suppose that the lemma is false and $G$ has a prosoluble subgroup $H$ for which the inequality $h(H)\leq4$ fails. Choose in $G$ an open normal subgroup $N$ such that $h(NH/N)\geq5$. Because $G$ is not torsion it follows that $G$ contains a $p$-element of infinite order. Choose a Sylow $p$-subgroup $P$ in $G$. Write $P_0=P\cap N$ and $K=N_G(P_0)$. In view of the Frattini argument, $G=NK$. Since $N$ is open, $P_0$ is not torsion. It follows from Lemma \ref{infpq3} that for each $r\neq p$ the Sylow $r$-subgroups of $K$ either are cyclic or (generalized) quaternion.

Next, choose in $K$ a subgroup $J$ with the property that $NJ=NH$. In view of \cite[Lemma 2.8.15]{rz} the subgroup $J$ can be chosen in such a way that $J\cap N$ is pronilpotent. Thus, $J$ is a prosoluble subgroup such that for each $r\neq p$ the Sylow $r$-subgroups of $J$ either are cyclic or (generalized) quaternion. Lemma \ref{44} tells us that every finite image of $J$ has Fitting height at most 4. On the other hand, the image of $J$ in $G/N$ has Fitting height at least 5. This is a contradiction.
\end{proof}

Denote by $\mathfrak X$ the class of all finite groups whose soluble subgroups are of Fitting height at most 4. The class $\mathfrak X$ is closed under taking quotients. Combining this fact with Lemma \ref{2244} we deduce that if $G$ is as in Theorem \ref{main2} and if $G$ is not torsion, then $G$  is a pro-$\mathfrak X$ group (cf. Lemma 3.8 in \cite{ppz}).
 
Recall that the nonsoluble length $\lambda(G)$ of a finite group $G$ is defined as the minimum number of nonsoluble factors in a normal series each of whose factors is either soluble or a nonempty direct product of non-abelian simple groups. It was shown in \cite[Cor.~1.2]{KhSh:15} that the nonsoluble length of a finite group $G$ does not exceed the maximum Fitting height of soluble subgroups of $G$. It follows that for any group $G$ in $\mathfrak X$ we have $\lambda(G)\leq4$. Therefore each group $G$ in $\mathfrak X$ has a characteristic series of length at most 24 each of whose factors is either nilpotent or a direct product of non-abelian simple groups. More precisely, each group $G$ in $\mathfrak X$ has a characteristic series
\[1=G_0\leq G_1\leq\dots\leq G_{24}=G\]
such that the factors $G_5/G_4$, $G_{10}/G_{9}$, $G_{15}/G_{14}$, $G_{20}/G_{19}$ are direct product of non-abelian simple groups while the other factors are nilpotent. We apply results of Wilson \cite[Lemma~2 and Lemma~3]{wilson} and deduce that any pro-$\mathfrak X$ group has a characteristic series of length at most 24 each of whose factors is either pronilpotent or a Cartesian product of non-abelian simple groups. Thus, we have proved the following result.

\begin{lemma}\label{series} Assume the hypotheses of Theorem \ref{main2} and suppose that $G$ is not torsion. Then $G$ has a characteristic series of length at most $24$ each of whose factors is either pronilpotent or a Cartesian product of non-abelian finite simple groups.
\end{lemma}

We are now ready to complete the proof of Theorem \ref{main2}.

\begin{proof}[Proof of Theorem \ref{main2}] Recall that $G$ is a profinite group in which each element has either finite or prime power (possibly infinite) order. In other words, whenever $x$ is an element of infinite order in $G$ there is a prime $p$ such that $C_G(x)$ is a pro-$p$ group. Assume that $G$ is not torsion. We need to show that $G$ is virtually pro-$p$ for some prime $p$.

Lemma \ref{series} tells us that $G$ has a finite series of characteristic subgroups all of whose quotients are either pronilpotent or Cartesian products of finite simple groups. We use induction on the length of that series. If the length is 1, the result follows from Lemmas \ref{nilp2} and \ref{cart2}. Assume that the length is at least 2. Thus, $G$ has a normal subgroup $N$ such that both $N$ and $G/N$ are either torsion or virtually pro-$p$, possibly for different primes $p$.

If both $N$ and $G/N$ are torsion, then $G$ is torsion, contrary to the assumptions.  

Suppose that $N$ is not torsion and thus $N$ is virtually pro-$p$ for some prime $p$. Set $P=O_p(N)$. Lemma \ref{infpq3} tells us  that for each $q\neq p$ the pro-$q$ subgroups of $G$ are finite, either cyclic or (generalized) quaternion. It follows that if $G/N$ is not torsion, then $G/N$ is virtually pro-$p$ for the same prime $p$. In this case $G$ is virtually pro-$p$ and we are done. If $G/N$ is torsion, we invoke Herfort's theorem \cite{herfort} on finiteness of the set of prime divisors of the order of a compact torsion group. Combining that theorem with the fact that for each $q\neq p$ the pro-$q$ subgroups of $G$ are finite we again conclude that $G/N$ is virtually pro-$p$ whence $G$ is virtually pro-$p$, as required.

Thus, we need to deal with the case where $N$ is torsion while $G/N$ is non-torsion, virtually pro-$q$. Wilson's Theorem \ref{wilson} tells us that $N$ has a finite series of characteristic subgroups $$N=N_1\geq N_2\geq\dots\geq N_s\geq N_{s+1}=1 $$ in which each factor either is a pro-$p$ group for some prime $p$ or is isomorphic (as a topological group) to a Cartesian product of isomorphic finite simple groups. We now use induction on the length $s$ of the above series and assume $G/N_s$ is non-torsion, virtually pro-$q$. Thus, without loss of generality we can assume that $N$ either is a pro-$p$ group for some prime $p$ or is a Cartesian product of isomorphic finite simple groups. If $N$ is pro-$q$, then $G$ is virtually pro-$q$ and we are done. 

Thus, assume that $N$ is not a pro-$q$ group and choose a subgroup $A\leq G$ isomorphic to $\Bbb Z_q$. If $A$ normalizes a nontrivial pro-$q'$ subgroup $B$, observe that $C_B(a)=1$ for each $a\in A^\#$. Thus, in view of Lemma \ref{resh} we obtain a contradiction. Therefore $N$ cannot be a pro-$p$ group and so $N$ is a Cartesian product of isomorphic finite simple groups. Note that in the natural action of $A$ on $N$ the group $A$ permutes the simple factors of $N$. Let $S\leq N$ be a simple factor. If the stabilizer of $S$ in $A$ is nontrivial, then we see that a subgroup of finite index in $A$ centralizes $S$ and, since $S$ contains $q'$-elements, we obtain elements of infinite order divisible by more than one prime. This is a contradiction so we assume that the stabilizer of $S$ in $A$ is trivial. Therefore the direct simple factors $S^a$ and $S^b$ are distinct whenever $a$ and $b$ are distinct elements of $A$. Choose a $q'$-element $x\in S$. It follows that the subgroup $\langle x^a;\ a\in A^\#\rangle$, that is, the subgroup generated by all $A$-conjugates of $x$ is an abelian pro-$q'$ subgroup normalized by $A$. We have already seen that the existence of such a subgroup leads to a contradiction. The proof is now complete. 
\end{proof}

\end{document}